\documentclass{article}

\usepackage{amsmath,amsfonts,amssymb,graphicx,mathtools}
\usepackage{amsthm,xspace}
\usepackage{enumitem}

\title{Genus 1 Curves in Severi--Brauer Surfaces}
\author{David J. Saltman\\
Center for Communications Research\\
805 Bunn Drive\\
Princeton, NJ 08540}

\date{May 13, 2021}

\newtheorem{theorem}{Theorem}
\newtheorem{proposition}{Proposition}
\newtheorem{lemma}{Lemma}
\newtheorem{corollary}{Corollary}

\begin{document}

\baselineskip=.3in

\maketitle

\def\newline{\hfill\break}
\def\scong{{\scriptstyle\|}\lower.2ex\hbox{$\wr$}}
\def\Z{{\mathbb Z}}
\def\R{{\mathbb R}}
\def\Q{{\mathbb Q}}
\def\A{{\mathbb A}}
\def\L{{\mathbb L}}
\def\C{{\mathbb C}}
\def\SB{\mathop{\rm SB}\nolimits}
\def\Ext{\mathop{\rm Ext}\nolimits}
\def\Aff{\mathop{\rm Aff}\nolimits}
\def\SAff{\mathop{\rm SAff}\nolimits}
\def\Jac{\mathop{\rm Jac}\nolimits}
\def\Tor{\mathop{\rm Tor}\nolimits}
\def\Div{\mathop{\rm Div}\nolimits}
\def\Prin{\mathop{\rm Prin}\nolimits}
\def\Sand{\mathop{\rm Sand}\nolimits}
\def\Cl{\mathop{\rm Cl}\nolimits}
\def\Tr{\mathop{\rm Tr}\nolimits}
\def\tr{\mathop{\rm tr}\nolimits}
\def\End{\mathop{\rm End}\nolimits}
\def\Bil{\mathop{\rm Bil}\nolimits}
\def\Res{\mathop{\rm Res}\nolimits}
\def\Im{\mathop{\rm Im}\nolimits}
\def\Hom{\mathop{\rm Hom}\nolimits}
\def\Ind{\mathop{\rm Ind}\nolimits}
\def\Norm{\mathop{\rm Norm}\nolimits}
\def\ord{\mathop{\rm ord}\nolimits}
\def\inf{\mathop{\rm inf}\nolimits}
\def\res{\mathop{\rm res}\nolimits}
\def\Map{\mathop{\rm Map}\nolimits}
\def\Pic{\mathop{\rm Pic}\nolimits}
\def\PGL{\mathop{\rm PGL}\nolimits}
\def\GL{\mathop{\rm GL}\nolimits}
\def\SL{\mathop{\rm SL}\nolimits}
\def\Gal{\mathop{\rm Gal}\nolimits}
\def\Br{\mathop{\rm Br}\nolimits}
\def\ram{\mathop{\rm ram}\nolimits}
\def\Spec{\mathop{\rm Spec}\nolimits}
\def\gcd{\mathop{\rm gcd}\nolimits}
\def\lcm{\mathop{\rm lcm}\nolimits} 
\def\deg{\mathop{\rm deg}\nolimits}
\def\rank{\mathop{\rm rank}\nolimits}
\def\Schur{\mathop{\rm Schur}\nolimits}
\def\Aut{\mathop{\rm Aut}\nolimits}
\def\Out{\mathop{\rm Out}\nolimits}
\def\Cor{\mathop{\rm Cor}\nolimits}
\def\CH{\mathop{\rm CH}\nolimits}
\def\M{{\cal M}}
\def\J{{\cal J}}
\def\S{{\cal S}}
\def\O{{\zeta}}
\def\diagram{\def\normalbaselines{\baselineskip21pt\lineskip1pt
	\lineskiplimit0pt}}
\def\rtimes{\mathop{\times\!\!{\raise.2ex\hbox{$\scriptscriptstyle|$}}}
	\nolimits} 
\def\blackbox{\hbox{\vrule width6pt height7pt depth1pt}} 
\outer\def\Demo #1. #2\par{\medbreak\noindent {\it#1.\enspace}
	{\rm#2}\par\ifdim\lastskip<\medskipamount\removelastskip
	\penalty55\medskip\fi}
\overfullrule=0pt
\def\Br{\mathop{\rm Br}\nolimits}
\def\Cor{\mathop{\rm Cor}\nolimits}
\def\A{{\mathbb A}}
\def\H{{\mathbb H}}
\def\P{{\mathbb P}}\hsize=5.9truein
\voffset=.5truein
\def\R{{\mathbb R}}
\def\hangbox to #1 #2{\vskip1pt\hangindent #1\noindent \hbox to #1{#2}$\!\!$}
\def\myitem#1{\hangbox to 30pt {#1\hfill}} 

\begin{abstract}
In a talk at the Banff International Research Station 
in 2015 Asher Auel asked questions about genus one curves 
in Severi-Brauer varieties $SB(A)$. More specifically he asked 
about the smooth cubic curves in Severi-Brauer surfaces, 
that is in $SB(D)$ where $D/F$ is a degree three 
division algebra.  Even more specifically, 
he asked about the Jacobian, $E$, of these curves. 
In this paper we give a version of an answer to both these questions, 
describing the surprising connection between these curves and 
properties of the algebra $A$.  
Let $F$ contain $\rho$, a primitive 
third root of one. Since $D/F$ is cyclic, it is 
generated over $F$ by $x,y$ such that 
$xy = \rho{yx}$ and we call $x,y$ a skew commuting 
pairs. The connection mentioned above is between 
the Galois structure of the three torsion points 
$E[3]$ and the Galois structure of skew commuting 
pairs in extensions $D \otimes_F K$. Given a description 
of which $E$ arise, we then describe, via Galois cohomology, 
which $C$ arise.

\end{abstract}

\clearpage

\Large

\section{Introduction}

In algebraic geometry, perhaps the simplest 
object is projective space. The next simplest object might be a Severi--Brauer variety, which is only interesting when 
the ground field $F$ is not algebraically closed. Severi--Brauer varieties are defined by a central simple algebra $A/F$. In fact, the Severi--Brauer variety, $\SB(A)$, 
is defined as the variety of minimal right ideals of $A$. More precisely, if $A/F$ 
has degree $n$ (i.e. dimension $n^2$ over $F$), then $\SB(A)$ is a closed subvariety of the 
Grassmann variety $G_n(A)$ consisting of the $n$ dimensional subspaces of $A$ which are right ideals. It follows that $\SB(A)$ 
has a rational point if and only if $A \cong M_n(F)$, or in words, if $A$ is split. 
If $A = \End_F(V)$ is split, then $\SB(A)$ is isomorphic to the projective space $\P^{n-1}$, since any such right ideal $I$ can be identified with the line $L \subset V$ which 
is the range of all nonzero elements of $I$. 

This explains why we view $\SB(A)$ as ``almost" projective space. After a finite extension of $F$, $\SB(A)$ is projective space. That is, 
$\SB(A)$ is a form of projective space.  

If $A/F$ has degree two, that is, $A/F$ is a quaternion algebra, then $\SB(A)$ is a curve 
and a form of the projective line. Quite 
a bit is known about these curves, since they are just the smooth conics in $\P^2$. 

Thus, perhaps, it pays to consider the case 
of next highest dimension, the Severi--Brauer 
variety of a degree-three division algebra $D/F$. Note that the division algebra case is the interesting one. If $A/F$ is degree three and not a division algebra, then $A$ is split and $\SB(A)$ is $\P^2$. 

Of course $\SB(D)$ is a form of $\P^2$, and to understand it one might want to understand the curves in $\SB(D)$. To start with, suppose $L \subset \SB(D)$ is a line, by which we mean it becomes a line when we split $D$. If we extend scalars to the separable closure $\bar F$, then $D \otimes_F \bar F = \End_{\bar F}(V)$ and the line 
$L$ is identified with a subspace $W \subset V$ of dimension 2. 
To such a $W$ we can associate the ideal, $J$, of all elements with range in $W$, which is of dimension 6 over $\bar F$. If $L$ 
is defined over $F$, then $J$ is Galois invariant and so yields 
an ideal of $D$ of dimension 6. This implies $D$ is split, 
the case we are avoiding. 

Suppose, then, that $C \subset SB(D)$ is an absolutely irreducible  smooth degree two curve or a conic. (The other cases are easy.) It follows that $C$ is smooth. $C$ has no rational point, so $C$ 
is of the form $\SB(D')$ for $D'/F$ a quaternion division algebra. That is, there is a stalk 
$R = {\cal O_{C}}$ of the sheaf of regular functions on $\SB(D)$ such that $R/M$ is the 
field of fractions, $K'$, of $\SB(D')$. 

Almost by definition, $\SB(D)$ splits $D$. Since $R$ is a stalk, $R$ splits $D$. By the functionality of the Brauer group, $K'$ 
splits $D$. But the kernel of $\Br(F) \to 
\Br(K')$ is generated by $D'$, implying 
that $D$ is trivial or Brauer equivalent 
to $D'$, and both are a contradiction. 
Thus $\SB(D)$ does not contain any 
absolutely irreducible conics. 

There is a better way to view the above result. Let $\bar F$ be the separable closure of 
$F$, so $\SB(D) \times_F \bar F = \P^2_{\bar F}$. 
The Picard group of $\P^2_{\bar F}$ consists of the 
standard line bundles ${\cal O}(n)$ for 
$n \in \Z$. Artin (\cite{A} p. 203) observed that the Picard 
group of $\SB(D)$ is generated by ${\cal L}$, 
the unique line bundle that pulls back to 
${\cal O}(3)$ on $\P^2_{\bar F}$. Viewing the Picard 
group as the divisor class group, we have that all curves on $\SB(D)$ have degree 
a multiple of three. 

If $C$ is an reducible curve of degree 3, 
then since no lines are Galois invariant 
it must be a union a three distinct lines which do 
not all intersect. That is, $C$ is a triangle determined by three distinct 
but Galois conjugate points. This means that $\SB(D)$ has plenty of triangles, namely, 
one for each maximal subfield. If $C$ is an absolutely irreducible but 
not smooth curve, then $C$ has one or two singular points which must be 
Galois conjugate and so $D$ is split by a degree one or two extension, which 
is impossible. 

Thus we can turn to the situation we are really interested in, 
where $C \subset \SB(D)$ is a cubic smooth absolutely irreducible curve, 
which is a 
curve of genus one. All this discussion motivates the question asked by 
Asher Auel, namely, to characterize the genus one curves in $\SB(D)$. 

It is not hard to see from basic considerations that 
the $j$ invariants that appear form a dense set in $\P^1_F$. 
Using the techniques of \cite{VV}, V\'arilly-Alvarado and Viray 
have given an explicit 
subset of such $j$ invariants which appear. 
Asher himself observed in unpublished work that cubic curves 
with $j = 0$ appear. The techniques developed here can distinguish 
between curves of the same $j$ invariant, which are not isomorphic. 
Instead of $j$ invariants, what is important for our main result 
is the group of three torsion points as a Galois module. 
Given that the elliptic curve $E$ appears as the Jacobian 
of a curve $C$, we describe which $C$ arise via Galois 
cohomology. Note the connection between these results 
and the result of \cite{O2} p. 331 which, for fixed $E$, 
gives the connection between cohomology and maps from 
genus one curves to Severi-Brauer varieties. The author 
would like to thank Asher Auel for the original question 
but most importantly for his detailed responses to earlier 
versions of this paper.

\section{Preliminaries} 

Let $A/F$ be a central simple algebra of degree three. Set $\bar F$ 
to be the separable closure of $F$, and $\bar G$ the Galois group 
of $\bar F/F$. Define $\SB(A)$ to be the Severi--Brauer 
variety of the central simple algebra $A$. Recall 
that $\SB(A)$ is the variety of right ideals of 
$A$ of dimension 3 over $F$. For simplicity 
we assume the characteristic of $F$ is not two or three. 

We know that $\SB(A)$ is a form of $\P^2$ and that 
the line bundle ${\cal O}(3)$ is defined on $\SB(A)$, 
which is also clear because this is the anticanonical bundle. 
We want to be more concrete, however. 
The tensor power $A^3 = A \otimes_F A \otimes_F A$ has an action 
of the symmetric group $S_3$ and we consider $B$, the algebra 
of $S_3$ fixed elements. By \cite{S}, this action of $S_3$ is induced 
by conjugation by a group $S_3 \subset (A^3)^*$. That is, 
$B$ is the centralizer of $F[S_3] \subset A^3$. But $F[S_3]$ 
is the direct sum $F \oplus F \oplus M_2(F)$ where the first 
$F$ corresponds to the trivial representation. If $e \in F[S_3]$ 
is the associated idempotent to this first summand, we set $S^3(A) = eBe$. Note that 
if we extend scalars to $\bar F$, so $A \otimes_F \bar F = 
\End_{\bar F}(\bar V)$, then $S^3(A \otimes_F \bar F) = 
\End_{\bar F}(S^3(\bar V))$, where $S^3(\bar V)$ is the 
symmetric power of the vector space $\bar V$. For this reason 
we call $S^3(A)$ the symmetric power of $A$. Note that 
this implies that $S^3(A)$ has degree 10. Also note that as 
$A^3 = M_{27}(F)$ is split, $S^3(A)$ is also split and so 
$S^3(A) = M_{10}(F)$. 

In particular, $\SB(S^3(A)) = \P^9$ is the nine dimensional 
projective space. There is an embedding 
$\SB(A) \to \SB(S^3(A))$ defined by $I \to e(I \otimes I \otimes I)^{S_3}e \subset 
S^3(A)$. Extending scalars to to $\bar F$ again, this is 
just the Segre embedding $\P^2(\bar F) \to \P^9(\bar F)$. 
Finally, the line bundle ${\cal O}(1)_{\P^9}$ is defined on 
$\P^9 = \SB(S^3(A))$. If $L$ is the restriction of this bundle 
to $\SB(A)$, then after extending scalars to $\bar F$, 
$L$ becomes ${\cal O}(3)_{\P^2}$ which shows this 
bundle is defined over $\SB(A)$. 

\begin{lemma} \label{lem:1}
 There is a line bundle $L$ 
on $\SB(A)$ which becomes ${\cal O}(3)_{\P^2}$ 
after scalar extension to $\bar F$ and is the restriction 
of ${\cal O}(1)_{\P^9}$ on $\SB(S^3(A))$. 
\end{lemma}

The bundle ${\cal O}_{\P^2}(3)$ has a ten dimensional 
space, $\bar W$, of global sections which are the cubic forms in three variables. 
That is, if $s \in \bar W$, then the zeroes of $s$ 
are a cubic curve in $\P^2$, and for all $s$ in a Zariski 
open set this is a smooth cubic curve of genus one. 
Now $L$ has a space of global sections of dimension 10, $W$, 
over $F$ and so a $s \in W$ defines a 
cubic curve $C$ in $SB(A)$. Once again, for a Zariski 
open set of $s$ this $C$ is smooth and of genus one. 
Of course, except in the trivial case $A$ is split, 
$\SB(A)$ has no rational points and so $C$ has no rational 
points. However, we can define $E(C)$ as the Jacobian 
of $C$ and $E(C)$ is an elliptic curve. The question 
we attack in this paper is the question of which 
$E(C)$ appear for a given $A$, and given $E(C)$ which $C$ appear. 

Of course the group $\PGL_3(\bar F)$ acts on $\P^2(\bar F)$ 
and this action extends to an action on $\P^9(\bar F)$. 
Viewing this as an action of $\GL_3(\bar F)$, this 
induces an action on the global sections of ${\cal O}(3)$ 
which is just the change of variables action on the 
cubic forms. 

We need to descend this group action to $F$. 
Let $\bar G$ be the Galois group of $\bar F/F$. 
Of course $A \otimes_F \bar F = M_3(\bar F)$ 
so there is a Galois action of $\bar G$ on $\GL_3(\bar F)$ 
such that the $\bar G$ fixed elements are $A^*$. 
Hilbert's Theorem 90 shows that $G$ fixed elements of 
$\PGL_3(\bar F)$ lift to $\bar G$ fixed elements of 
$GL_3(\bar F)$, so $A^*/F^*$ is the group of 
$\bar G$ fixed elements of $\PGL_3(\bar F)$. 
Moreover, $A^*/F^*$ is the group of $F$ automorphisms 
of $A$ and hence $A^*/F^*$ acts on $\SB(A)$ and 
$\SB(S^3(A))$. Of course, after extending scalars to 
$\bar F$, this is just the action of $PGL_3(\bar F)$ on 
$\P^2$ and $\P^9$. 

\section{Classical Invariant Theory}

Let us review the classical invariant theory 
of cubic curves over $\bar F$. This material 
is very far from new, but it is useful to review so we can treat the non-algebraically closed case. 
As a general reference we suggest \cite[Chapter 3]{D}. 

Let $\bar T$ be the projective space of $\bar W$, that is, 
$\bar T = \P^9$. Then $\bar T$ can be viewed as the space of cubic 
curves. We begin by recalling some basic facts about 
a smooth cubic curve. Any line in $\P^2$ intersects 
$\bar C$ in three points, counting multiplicity. 
Recall that a point $P$ on $\bar C$ is called an inflection 
point if the tangent line to $\bar C$ at $P$ has an intersection 
with $\bar C$ that is at least multiplicity 3 (any tangent line 
has multiplicity 2). For the moment fix one inflection point 
$[0]$ on $\bar C$. Looking at all the intersections 
of lines with $\bar C$, we see that the embedding 
$\bar C \to \P^2$ is defined by the line bundle 
${\cal L}(3[0])$. 

The Jacobian $E(\bar C)$ of $\bar C$ is the group of 
degree 0 divisors modulo principal divisors. 
In particular, $E(\bar C)$ is an abelian group. 
If we fix $[0]$ as above, then $\bar C \cong E(\bar C)$ 
via $[P] \to [P] - [0]$. 
Moreover, there is a natural action $E(\bar C) \times \bar C
\to \bar C$ in which, for example, it is true 
that $[P] - [0]$ acting on $[0]$ is $[P]$. 
Thus $E(\bar C)$ is a subgroup of the automorphism 
group of $\bar C$ which we call the translations 
and $E(\bar C)$ acts transitively on $\bar C$. 

Since $\bar C \cong E(\bar C)$ a choice of $[0]$ 
imparts an additive group structure on $\bar C$. 
It will be useful to recall the classical construction 
of this structure. 
If $P,Q,R$ are three points 
on $\bar C$ then we say these points "sum to 
zero" if they are colinear and hence the intersection 
points of some line with $\bar C$. To make a group 
with zero $[0]$, for any point $[P]$, 
one defines $[-P]$ by the condition that 
$[P]$, $[-P]$, and $[0]$ are colinear. 
Then $[P] + [Q] = [R]$ if $[P]$, $[Q]$, and 
$[-R]$ are colinear. Thus "summing to zero" 
is independent of the choice of $[0]$. 

The automorphisms of $\bar C$ have the form 
$T \rtimes G$ where $T = E(\bar C)$ is the group of translations and 
$G$ is the group of automorphisms that fix $[0]$. 
Thus if $\bar C$ and $\bar C'$ 
are two isomorphic smooth cubic curves in $\P^2$, 
there is an isomorphism between them that preserves the embedding 
implying that there is an isomorphism between them 
that extends to $\P^2$. That is, $\bar C$ and $\bar C'$ 
are in the same $PGL_3(\bar F)$ orbit in $\P^9$. 
The quotient $\P^9/\PGL_3(\bar F)$ is birationally 
$\P^1$ and we will later remind the reader that there is 
an induced map map to the $j$ line which is generically 12 to 1. 
The action of $PGL_3(\bar F)$ on cubic curves will be key to 
our argument. 

Let $S^+ \subset PGL_3(\bar F)$ 
be the stabilizer of some $\bar C$. 
Then $S^+$ consists of the automorphisms of $\bar C$ 
which preserve the embedding in $\P^2$. 
All the elements of $G$ preserve the embedding, but only 
the elements of $T$ of order 3 preserve the embedding. 
It follows that $S^+ = S \rtimes G$ where $S$ is isomorphic to 
$\Z/3\Z \oplus \Z/3\Z$ and is the group of translations by 
three torsion elements of $E(\bar C)$. 
If the $j$ invariant of $\bar C$ 
is not $0$ or $1728$, then $G = S^+/S = \Z/2\Z$ and $G$ 
is generated by the $-1$ map on $C$ viewed as an elliptic curve 
with $[0]$ as the zero element.  
If $j = 0$ 
then $S^+/S = C_6$ the cyclic group of order 6, and if $j = 1728$ then 
$S^+/S = C_4$.  

Given $\bar C$, we let $I$ be the set of its 
nine inflection points. In the following discussion 
we have two goals. First, we observe that we can 
define a nine point "Jacobian", $E(I)$, for $I$ that does not 
involve $\bar C$. We accomplish this 
by restricting the argument reviewed above 
for all of $\bar C$. Second, we want to relate this 
$E(I)$ to the stabilizer of $\bar C$ in $\PGL_3(\bar F)$ 
and, in particular, to the group of translations. 
What we will show is that if we pick one point 
in $I$ to be the identity, then the other 8 points 
correspond to  
points on the elliptic curve $E(\bar C)$ of order 3. 
The next result describes some properties 
of $I$ where we assume some $\bar C$ exists but not anything 
specific about it. 

\begin{lemma} \label{lem:2}
Suppose $P$, $Q \in I$ are distinct. 
Then there is exactly one third point $
R \in I$ such that $P$, $Q$, $R$ 
are colinear in $\P^2$ and $R$ is distinct from both 
$P$ and $Q$. 
\end{lemma}

\begin{proof} We identify $\bar C$ with $E(\bar C)$ via 
a choice of $[0]$. 
A line in $\P^2$ meets any $\bar C$ in three points. 
Thus, $R \in \bar C$ is well defined. Moreover, 
$P + Q + R$ is zero in $\bar C$, and so $R$ must also be a 
three torsion point and hence an inflection point and hence in 
$I$. If, say, $R = Q$ then $P + 2Q = 0$ in $\bar C$ 
which implies $P = Q = R$.
\end{proof}

Because of the above we say three points of $I$ are colinear if they are distinct 
and colinear in $\P^2$, or are all the same. 

We can form a group from this relationship on $I$, 
by repeating for $I$ the construction of $E(\bar C)$ 
but restricted to these nine points. 
Let $E(I)$ be the set of equivalences classes 
of all pairs $(P,Q)$, $P,Q \in I$, where we say $(P,Q) \sim (R,S)$ 
if there is a single $T \in I$ such that $P,S,T$ and 
$R,Q,T$ are colinear. Note that all pairs $(P,P)$ are equivalent.  
We let $\{(P,Q)\}$ be the equivalence class containing $(P,Q)$. 

\begin{lemma} \label{lem:3}
\begin{enumerate}[label=(\alph*)]

\item If 
$(P,Q) \sim (R,S)$ then $(P,R) \sim (Q,S)$. 

\item If $(P,Q) \sim (R,Q)$ then $P = R$. 

\item For any $P,Q,R \in I$ there are unique $S, S' \in I$ 
such that $(P,Q) \sim (R,S) \sim (S',R)$. 

\item If $(P,Q) \sim (Q,P)$ then $P = Q$. 

\item If we fix $Q$ then the map $P \to \{(P,Q)\}$ is a bijection 
between $I$ and $E(I)$. 
\end{enumerate}
\end{lemma}

\begin{proof} 
Part a) is direct from the 
definition. For b), we have $P,Q,T$ and $R,Q,T$ 
colinear and so $P = R$.
Turning to c), we have $R,Q,T$ colinear for a unique $T$ and define 
$S$ by taking $P,T,S$ colinear (defining $S'$ is similar). 
For d), only $P$ is colinear with $P$ and $P$ 
($P$ is an inflection point). Part e) follows 
from a) and b). 
\end{proof}

We can now proceed to define the group structure on $E(I)$. 
We say $\{(P,Q)\} + \{(R,S)\} = \{(P,U)\}$ if $(R,S) \sim (Q,U)$. 
One can check this is well defined with identity the class of $(P,P)$.  
Also $\{(P,Q)\} + \{(Q,R)\} = \{(P,R)\}$ 
so $\{P,Q\} +\{Q,P\} = \{P,P\}$. Since colinearity is a symmetric 
relationship this defines an abelian group. Finally, and 
importantly, there is a 
natural action of $E(I)$ on $I$ defined by $P + \{(Q,P)\} = Q$. 

\begin{lemma} \label{lem:4}
\begin{enumerate}[label=(\alph*)]
\item If $\alpha \in E(I)$ then $3\alpha = 0$.

\item If $\alpha \in E(I)$ then $(P,Q) \sim (P + \alpha, Q + \alpha)$.

\item $P,Q,R$ are colinear if and only if for any (hence all) S, 
$\{(P,S)\} + \{(Q,S)\} + \{(R,S)\} = 0$. In particular, 
$P$, $Q$, and $R$ are colinear if and only if $(Q,P) \sim (P,R)$. 

\item If $P,Q,R$ are colinear then so is $P + \alpha, Q + \alpha, R + \alpha$ 
for any $\alpha \in E(I)$. 

\item If $P \in I$ and $\alpha \in E(I)$, then 
$P$, $P + \alpha$ and $P + 2\alpha$ are colinear. 
\end{enumerate}
\end{lemma}

\begin{proof}
To begin with a), if $P,Q,R$ are colinear then 
$(P,Q) \sim (Q,R) \sim (R,P)$ and $3\{P,Q\} = \{(P,Q)\} + \{(Q,R)\} + \{(R,P)\} =  \{P,P\}$. 
As for b), if 
$\alpha$ contains $(S,P)$ and $(T,Q)$ then $(S,P) \sim (T,Q)$ 
so $(P,Q) \sim (S,T)$.  
Turning to c), the independence from the choice of $S$ is clear from 
adding $3\{(S',S)\}$ to both sides. Thus we need only consider 
$\{(P,P)\} + \{(Q,P)\} + \{(R,P)\} = 0$ or 
$\{(Q,P)\} + \{(R,P)\} = 0$ or $(Q,P) \sim (P,R)$ which 
means the line through $Q$ and $R$ also goes through 
$P$. Part d) is now clear. As for 3), let $\alpha = \{(Q,P)\} 
= \{(R,Q)\}$ so $P,Q,R$ are colinear. Then $P + \alpha = 
Q$ and $2\alpha = \{(Q,P)\} + \{(R,Q)\} = \{(R,P)\}$ 
so $P + 2\alpha = R$. 
\end{proof}

Let us, briefly, assume $F$ is not necessarily equal to $\bar F$. 
Suppose $C \supset I$ is a cubic curve defined over $F$ with Jacobian $E$ and 
$I$ are the inflection points. 
Then $E \times C \cong C \times C$ via the morphism $(\alpha,P) \to (P + \alpha, P)$. 
Thus there is a projection $C \times C \to E$ that maps $(P,Q) \to [P] - [Q]$ 
where $[P] - [Q]$ is the degree-zero divisor on $C$. The diagram:
$$
\begin{matrix}
I \times I &\subset& C \times C\\
\downarrow&&\downarrow\\
E(I)& \subset & E(C)
\end{matrix}$$
commutes. All of this makes sense over $\bar F$ where $I \times I$ 
is a zero dimensional reduced variety. Of course, $I \times I$ is defined over 
any field where $I$ is defined. The above diagram imparts to 
$E(I)$ a variety structure defined over any field $F$ where 
$C$ and hence $I$ are defined, though of course when $F \not= \bar F$, 
then $K_1 = \Spec(E(I))$ and $K_2 = \Spec(I \times I)$ are direct sums of fields, 
perhaps strictly containing $F$. 
If $I$ is defined over $F$ but $C$ is not, we can proceed as follows. 
The $\bar F$ points of $I \times I$ are morphisms $\phi_{P,Q}:K_2 \to \bar F$ and there 
is one for each $P,Q$. Then we can define $K_1$ to be $\{f \in K_2 | 
\phi_{P,Q}(f) = \phi_{R,S}(f)\ $whenever$\ (P,Q) \sim (R,S)\}$. 
Since this is an linear condition on $f$ this is compatible with the 
definition of $K_1$ over $F$'s where $C$ is not defined. All of which 
is saying:

\begin{lemma} \label{lem:5} 
$I \times I \to E(I)$ is a morphism of zero dimensional varieties whenever $I$ is defined. 
\end{lemma}

\begin{corollary} \label{cor:6}
Suppose $\bar C'$ also contains $I$. 
Then $I$ forms the set of inflection points of $\bar C'$ also. 
$E(I)$ can be identified with the three torsion points of both 
the Jacobians of $\bar C$ and $\bar C'$. 
\end{corollary}

We know classically that the action of $E(I)$ on $I$ 
extends first to an action on $\bar C$, namely translation 
by three torsion points, and further therefore to an action on 
$\P^2$. Thus $E(I)$ corresponds to a subgroup $S \subset \PGL_3(\bar F)$. 
We can give more detail about $S$, and this is well known 
(see \cite{O1} p. 131). Let $F_3$ be the field 
of three elements. 
From the definition $S$ is generated by $x,y$ each of order three 
which commute in $PGL_3(\bar F)$ and represent a choice of 
basis, $\alpha, \beta$ over $F_3 = \Z/3\Z$ of $E(I)$. 
Let $\tilde x, \tilde y \in GL_3(\bar F)$ be preimages of $x,y$ 
where we can assume (over $\bar F$), that 
$\tilde x^3 = \tilde y^3 = 1$. Up to conjugation, there are two 
possibilities. Either $\tilde x$ and $\tilde y$ commute or 
$\tilde x{\tilde y} = \rho\tilde y\tilde x$ for a nontrivial 
$3$ root of one $\rho$. We will argue the that later holds, a 
classical fact related to the existence of the Weil pairing. 

Fix $P \in I$. Then $P, P + \alpha, P + 2\alpha$ are colinear;
as are $P + \beta, P + \beta + \alpha$, and $P + \beta + 2\alpha$ 
as well as $P + 2\beta, P + 2\beta + \alpha, P + 2\beta + 2\alpha$. 
All of these lines are left invariant by $\alpha$. 
Note that these are three distinct lines because any line contains 
at most three points of $I$. Lifting to $GL_3(\bar F)$ acting on 
$\A^3$, there are three distinct two-dimensional subspaces 
preserved by $\tilde x$. But $\tilde x$ has three distinct eigenvalues 
and so these are precisely all the two-dimensional subspaces 
preserved by $\tilde x$. We can make the same argument 
for $\tilde y$ and conclude that $\tilde y$ 
preserves the spaces associated to the 
lines $P, P + \beta, P + 2\beta$, as well as 
$P + \alpha, P + \alpha + \beta, P + \alpha + 2\beta$ and 
$P + 2\alpha, P + 2\alpha + \beta, P + 2\alpha + 2\beta$. 
Note that the spaces preserved by $\tilde x$ and $\tilde y$ 
are independent of the choice of $\tilde x$, $\tilde y$ and 
only depend on $x$ and $y$. 

\begin{lemma} \label{lem:7}
 Fix a third root of unity $\rho$. After 
perhaps changing $y$ to $y^2$, we have 
$\tilde x\tilde y = \rho\tilde y\tilde x$. The action of $\tilde y$ 
permutes the eigenspaces of $x$ and vice versa. 
If $V_1,V_2,V_3$ are the two-dimensional spaces preserved by $\tilde x$ and $W_1,W_2,W_3$ 
the same for $\tilde y$, then points of $I$ correspond 
to the nine intersections $V_i \cap W_j$. $S$ acts transitively 
on $I$. 
\end{lemma}

\begin{proof} This is well known and appears in 
\cite{O1} p. 131. We include a proof because the details of the proof 
will be useful. If $\tilde x$, $\tilde y$ commuted they would have the same 
eigenspaces and hence the same preserved two-dimensional subspaces. 
This proves the first statement. The second is now clear and the 
third statement just says, as observed above, that each of the lines 
preserved by $x$ and $y$ share one point of $I$. It is immediate 
that $S$ acts transitively.
\end{proof} 

Thus associated to $I$ is a subgroup $S \subset PGL_3(\bar F)$ 
where $S$ is the image of $\tilde S \subset GL_3(\bar F)$ 
and $\tilde S$ is generated by $\tilde x$, $\tilde y$ 
with $\tilde x^3 = 1$, $\tilde y^3 = 1$, and $\tilde x\tilde y = 
\rho\tilde y\tilde x$. All such $\tilde S$ are conjugate in 
$GL_3(\bar F)$, and so all such $S$ are conjugate in $PGL_3(\bar F)$. 

The first consequence of Lemma~\ref{lem:7} is that $S$ has a pairing $$S \times S \to <\rho>$$ 
defined by the map $(x,y) \to \rho^i$ where $x,y$ lift to $\tilde x, \tilde y \in GL_3$ 
and $\tilde x\tilde y = \rho^i\tilde y\tilde x$. It follows that the 
preimage of $S$ in $GL_3(\bar F)$ contains the Heisenberg group of order 27. 
That is, if we choose $\tilde S$ to be generated by preimages $\tilde x$, 
$\tilde y$ with $1 = \tilde x^3 = \tilde y^3$, then $\tilde S$ is the Heisenberg 
group. 

The identification of $S$ with $E(I)$ via their action on $I$ 
therefore amounts to an isomorphism $\psi: S \to E(I)$ where $\psi(s)$ is the equivalence 
class containing $(s(P),P)$. We have: 

\begin{lemma} \label{lem:8}
 $\psi$ is well defined and is a group isomorphism. 
Moreover, as a subset of $PGL_3$, $S$ induces the trivial action on $E(I)$. 
\end{lemma}

A consequence of Lemma~\ref{lem:8} is that $E(I)$ has a pairing inherited from $S$. 
We quote a result from \cite[Chapter 3]{D}. 

\begin{theorem} \label{classical} 
The stabilizer of $I$ in $\PGL_3(\bar F)$ has the form 
$$S \rtimes SL_2(F_3).$$ 
\end{theorem} 

To understand the above theorem set $\Aff(I)$ to be the group of bijections of $I$ 
which preserve colinearity, and we view $E(I) \subset \Aff(I)$ 
via the action of $E(I)$ on $I$, which 
preserves colinearity by Lemma~\ref{lem:4}. 

\begin{lemma} \label{lem:10}. There is an exact sequence 
$$0 \to E(I) \to \Aff(I) \to \
GL_2(F_3) \to 0$$
where $\Aff(I) \to GL_2(F_3)$ is induced by taking the action on 
$E(I)$. 
\end{lemma}

\begin{proof} 
Let $\phi: I \to I$ be a bijection that preserves the colinearity 
relationship. Then $\phi$ induces an automorphism (i.e., element 
of $GL_2(F_3)$) of $E(I)$. This defines the morphism above. 
If $\phi$ is the identity on $E(I)$, then $(P,Q) \sim (\phi(P),\phi(Q))$ 
for all $P,Q$ so $(\phi(P),P) \sim (\phi(Q),Q)$ and there is an 
$\alpha \in E(I)$ containing all $(\phi(P),P)$. 
Thus $\phi(P) = P + \alpha$. Finally, if $A \in GL_2(F_3)$ and 
$P \in I$ we define $A_P: I \to I$ by letting $A_P(Q)$ be the 
element of $I$ such that 
$(A_P(Q),P) \in A(\{(Q,P)\})$. By the lemma above $A_P$ preserves 
colinearity and shows that there is a splitting $GL_2(F_3) \to 
\Aff(I)$ for every choice of $P$. In particular the map is onto.
As we saw in Lemma~\ref{lem:4}, the action of $E(I)$ on $I$ induces the 
trivial action on $E(I)$. 
\end{proof} 

Soon we will also need to consider $\SAff(I) \subset 
\Aff(I)$ which is the preimage of $SL_2(F_3)$. 

It will be useful to notice that the action of $\Aff(I)$ 
on $I$ induces a canonical element of $\gamma \in H^1(\Aff(I), E(I))$. 
Define $d: \Aff(I) \to E(I)$ by setting 
$d(g) = \{(g(P),P)\}$ for a choice of $P$. Then 
for $g,h \in \Aff(I)$, $d(gh) = \{(gh(P),P)\} = 
\{(gh(P),g(P))\} + \{(g(P),P)\} = g(\{(h(P),P)\} + \{(g(P),P)\} = 
g(d(h)) + d(g)$ and thus $d$ is a 1 cocycle. 
If we define $d'(g) = \{(g(Q),Q)\}$ then 
$d(g) - d'(g) = g(\{(P,Q)\}) - \{(P,Q)\}$
and $d$ and $d'$ are cohomologous and any choice 
of $P$ yields the same cohomology class we call 
$\gamma$. If we restrict 
$d$ to $E(I)$ then $d$ is the identity, viewed as an 
element of $\Hom(E(I),E(I)) = H^1(E(I),E(I))$. 
Of course if $\gamma_S$ is the restriction of $\gamma$ 
to $\SAff(I)$, then $\gamma_S$ also restricts to the 
identity on $E(I)$.

\begin{proposition} \label{gamma} $\gamma \in H^1(\Aff(I), E(I))$ is the 
unique element which restricts to the identity of 
$H^1(E(I),E(I)) = \Hom(E(I),E(I))$. 
The same holds for $\gamma_S \in H^1(\SAff(I),E(I))$. 
\end{proposition} 

\begin{proof} Let $G = GL_2(F_3)$ or $SL_2(F_3)$ 
and $H = \Aff(I)$ or $\SAff(I)$ so $H \to G$ has kernel $E(I)$. 
The Hochschild-Serre spectral sequence 
shows that there is an exact sequence 
$H^1(G,E(I)) \to H^1(H,E(I)) \to 
H^1(E(I),E(I))$ and so it suffices to show that 
$H^1(G,E(I)) = H^1(G,F_3 \oplus F_3) = 0$. 
Since $SL_2(F_3)$ has index 2 in $GL_2(F_3)$, it suffices 
to show $H^1(SL_2(F_3),E(I)) = 0$. 
As a finite group, one can calculate that $SL_2(F_3) = 
Q \rtimes F_3^+$ where $Q$ is the quaternion group of 
order 8. Since $H^1(Q,F_3 \oplus F_3) = 0$ and 
$(F_3 \oplus F_3)^Q = 0$, we are done by another use of 
Hochschild-Serre. 
\end{proof}

Now let $H \subset \PGL_3(\bar F)$ be the stabilizer of $I$. 
We have a homomorphism $\eta: H \to \Aff(I)$. 
Furthermore, any $\phi \in \PGL_3(\bar F)$ 
which is the identity on $I$ must be trivial. 
Thus $\eta$ is an injection. $S \subset H$ acts trivially 
on $I$ and hence $\eta(S) = E(I)$. 
We saw above that $E(I)$ has a pairing which must be 
preserved by $H$ and so $\eta(H) \subset \SAff(I)$. 
The above theorem amounts to saying that 
$\eta(H) = \SAff(I)$. 

If $S^+$ is the stabilizer of $\bar C$, then it must 
stabilize $I$ and hence $S^+ \subset H$ while 
$S \subset S^+$ is precisely the kernel of the action 
of $S^+$ on $E(I)$. 

Let ${\cal I}$ be the set of nine point subsets of 
$\P^2$ which are the inflections points of smooth cubic 
$\bar C$. Let ${\cal S}$ be the set of nine element subgroups 
of $PGL_3(\bar F)$ where any $S \in {\cal S}$ is the image 
of $\tilde S \subset GL_3(\bar F)$ and $\tilde S$ is 
generated by $\tilde x$, $\tilde y$ with $\tilde x^3 = 1$, 
$\tilde y^3 = 1$, and $\tilde x\tilde y = \rho\tilde y\tilde x$. 

\begin{theorem} \label{theom:9}
 Th above discussion yields a one to one correspondence 
between ${\cal I}$ and $\cal S$. 
\end{theorem}

\begin{proof} 
We saw above that $S$ determines $I$, once we observe that the 
description there is independent of the basis of $S$ chosen. 
Since all the elements of ${\cal S}$ are conjugate, the 
subset $I$ must be the inflection points of some smooth 
cubic $\bar C$. 
For the converse, $I$ determines its stabilizer $H$, and 
$S$ is the kernel of the action of $H$ on $E(I)$.
\end{proof}

Note that if $\bar G$ is the Galois group of $\bar F/F$, then the above correspondence clearly commutes with the action of $\bar G$. 

If follows from Theorem~\ref{theom:9} that if $\phi \in PGL_3(\bar F)$,  then 
$\phi{S}\phi^{-1} = S$ if and only if $\phi(I) = I$. 
We have

\begin{theorem} \label{theom:11}
$H$ is the normalizer of $S$ in $\PGL_3(\bar F)$.  
\end{theorem}

We can add to our understanding of Theorem~\ref {classical} as follows. 
We saw above that if $\phi$ normalizes $S$ then 
$\phi(I) = I$ and so $\phi$ maps to $\SAff(I)$.  
Conversely suppose if $a,b,c,d \in F_3$ are such that $ad - bc = 1$. 
Write 
$x_1 = \tilde x^a\tilde y^b$, $\tilde y_1 = \tilde x^c\tilde y^d$ 
and note that $\tilde x_1^3 = \tilde y_1^3 = 1$ and $\tilde x_1\tilde y_1 = 
\rho\tilde y_1\tilde x_1$. Thus setting $\tilde x \to \tilde x_1$ and 
$\tilde y \to \tilde y_1$ defines an automorphism of 
$M_3(\bar F)$ and $\phi$ exists by Noether-Skolem. Thus 
$H$ maps onto $\SAff(I)$.

For a given $I$, we can look at all the curves $\bar C$ 
which contain $I$. It is classically known that this forms 
a line $L_I$ in $\P^9$ and all $j$ invariants appear somewhere 
on a point of this line. It is instructive to outline how we know this. 
First of all, let $\tilde N \subset GL_3(\bar F)$ be in the preimage of 
$H = N(S)$. Let $S = <x> \oplus <y>$ and let $x',y' \in \tilde N$ 
be preimages with $x'^3 = 1 = y'^3$. If $\tilde S \subset \tilde N$ 
is the group generated by $x',y'$ then $\tilde S$ is the Heisenberg 
group and $\tilde S\bar F^*$ is the full preimage of $S$. 
Thus $\tilde N/\tilde S\bar F^* = N(S)/S = \SL_2(F_3)$. 
Let $M \subset \tilde H$ be the abelian group generated by 
$x'$ and $\rho$.  

Of course, $\bar F^3$ is a module over $\tilde N$. 
Obviously one could 
describe this module but we do not need this. It suffices to note that as a 
$\tilde S$ module it is $\Ind_M^{\tilde S} L$ where 
$M$ acts on $L = \bar F^1v$ by setting 
$x'(v) = \rho{v} = \rho(v)$. That is, $V = \bar F^3$ has a basis 
$\{v_0,v_1,v_2\}$ 
where $\bar Fv_i$ is an $x'$ eigenspace with eigenvalue $\rho^i$ and 
$y$ permutes the $v_i$. Let $V^*$ be the dual of $V$ with dual basis 
$v_0^*,v_1^*,v_2^*$. 
Of course $S^3(V^*)$ is spanned by the degree-three monomials in the $v_i^*$. 
Moreover, the inverse image of $I$ in $V$ are the nine lines $L_{ij}$ 
where $L_{ij}$ is the simultaneous zero of $v_i^*$ and 
$w_j^* = v_0^* + \rho^jv_1^* + \rho^{2j}v_2^*$ which is the eigenvector for $y$. 

Of course $\rho$ acts trivially on $S^3(V^*)$, and the $\tilde S$ fixed subspace is 
spanned by $v_0^*v_1^*v_2^*$ and $(v_0^*)^3 + (v_1^*)^3 + (v_2^*)^3$. 
Thus the projective space of $C$'s fixed by $\tilde S$ is a projective line. This is the classical Hessian pencil and 
associated Hessian normal form. 

Having made this computation, we are ready to prove: 

\begin{theorem} \label{theom:12}
 Suppose $S \subset PGL_3(\bar F)$ is associated with 
$I \subset \P^2$. Then a cubic curve $\bar C$ contains $I$ if and only 
if $S$ fixes $\bar C$ or equivalently that $\tilde S$ fixes $\bar C$, 
where $\tilde S$ is the Heisenberg group 
in the preimage of $S$ in $GL_3(\bar F)$. 
\end{theorem}

\begin{proof} 
If $\tilde S$ fixes $\bar C$, then $\bar C$ is in the span of 
$v_0^*v_1^*v_2^*$ and $(v_0^*)^3 + (v_1^*)^3 + (v_2^*)^3$. 
We need to show that any such $\bar C$ is in the ideal $(v_i^*,w_j^*)$ 
for any $i$ and $j$. This implies that $I \subset \bar C$. 
We compute that $w_0^*w_1^*w_2^* = 
(v_0^*)^3 + (v_1^*)^3 + (v_2^*)^3 + (-3)(v_0^*v_1^*v_2^*)$ 
and this direction is obvious. 

Conversely, suppose $I \subset \bar C$. 
We need to show that $\tilde S$ fixes $\bar C$. 
Note that any zero set of an $f$ contains $I$ if and only if 
$f \in \cap_{i,j} (v_i^*,w_j^*) = J$. Clearly $J$ and hence $J \cap S^3(V^*)$ 
is preserved by $\tilde H$. It thus suffices to show that no eigenvector 
for $\tilde H$, with nontrivial eigenvalue, is in $J$. We can check this 
one by one. $(v_0^*)^3 + {\rho}(v_1^*)^3 + {\rho^2}(v_2^*)^3$ 
is clearly not in $(v_0^*,w_0^*)$ and the same holds for 
$(v_0^*)^3 + {\rho^2}(v_1^*)^3 + {\rho}(v_2^*)^3$. 
The $\tilde H$ span of $(v_1^*)^3v_2^*$ has basis this vector and 
$(v_2^*)^2v_0$, $(v_0^*)^2v_1^*$. Any eigenvector has a nonzero 
$(v_1^*)^3v_2^*$ term and since $(v_1^*)^3v_2^* \notin (v_0^*,w_0^*)$ 
it follows that this eigenvector is not in $I$. 
The same argument works for the $\tilde S$ span of $(v_1^*)^2v_0^*$.
\end{proof}

We defined $L_I$ to the line of $\bar C$'s which 
contain $I$ or equivalently which are fixed by $S$. 
Ir follows that the normalizer, $N(S)$, of $S$ in 
$PGL_3(\bar F)$ acts on $L(I)$. Since all the $S$'s and 
$I$'s are conjugate, any $PGL_3(\bar F)$ orbit intersects 
$L_I$. 
If $\bar C, \bar C' \in L_I \cap O$  and 
$\bar C' = g(\bar C)$ then both these cubics have stabilizer 
containing $S$. Thus $g \in N(S)$. The quotient 
of $L_I$ by $N(S)$ is a $\P^1$, and the associated map 
$L_I \to \P^1$ can be viewed as taking any curve $\bar C$ 
to its $j$ invariant. A generic element of $L_I$, specifically 
one with $j \not= 0, 1728$, as stabilizer $S^+$ where 
$S^+/S$ has order 2. Thus $L_I \to \P^1$ is generically 
a Galois cover with group $N(S)/S^+ \cong SL_2(F_3)/{\pm 1}$ 
which has order 12. 
However we need to be more careful. Assume first that $\bar C$ 
has $j \not= 0,1728$. 
While there are 12 
cubic curves $\bar C \supset I$, there are 24 associated 
maps $\iota: I \to \bar C$ and $b: \bar C \to \P^2(\bar F)$, 
because if $s \in S^+$, we can replace $\iota$ by $s\iota$ and $b$ by 
$bs^{-1}$, preserving $b\iota$ which is the fixed embedding of $I$. 
Of course when $j = 0$ there are $4$ curves $\bar C$ and 
when $j = 1728$ there are $6$ such curves $\bar C$, but there 
are always $24$ embeddings and $24$ maps $\iota: I \to \bar C$. 
Thus with fixed embedding $I \to \P^2$, choosing an embedding 
$\bar C \to \P^2$ is also specifying an embedding $\iota: I \to \bar C$.  

We defined $\psi: S \to E(I)$ above. Fix an embedding 
$I \subset \P^2_{\bar F}$. Assume $\bar C \supset I$. 
We want to extend $\psi$ to an identification of $S$ and 
$E(I)$ with the three torsion points of $E(\bar C) = E$, 
the Jacobian of the genus one curve $\bar C$. 
Given $\iota: I \to \bar C$, 
there is an induced $E(\iota): E(I) \to E$ given by 
$(P,Q) \to P - Q$. If $S^+ \subset \PGL_3(\bar F)$ 
is the stabilizer of $\bar C$ we remarked above 
that $S^+/S$ was the automorphism group of 
$E(\bar C)$ because the translation action of $S$ on $\bar C$ becomes trivial on $E(\bar C)$. Given a fixed $\iota$, 
we define $\phi_C: S \to E(\bar C)$ as $E(\iota) \circ \psi$. That is, $\phi_C(s) = s(P) - P$ for any choice 
of $P$ in $I$. Let $\eta: S^+/S \cong \Aut(E(\bar C))$ be the isomorphism 
mentioned above. 

\begin{lemma} \label{lem:13}
 Assume $\bar C \subset \P^2_{\bar F}$. Then there is an induced 
injective 
group homomorphism 
$\phi_C = E(\iota) \circ \psi: S \to E(I) \to E(\bar C) = E$. 
If $t \in S^+$, then $\phi_C(tst^{-1}) = \eta(t)(\phi_C(s))$. 
All choices of $\phi_C$ have the form $\phi_C'(s) = \phi_C(tst^{-1})$. 
In particular, if $j \not= 0, 1728$, only 
$\phi_C$ and $-\phi_C$ arise. 
\end{lemma}

\begin{proof} All we have to show is that $\phi_C$ 
is an injective homomorphism. If $\phi{\bar C}(s)$ is the 
identity then $s(P) = P$ for one and hence all $P$. 
We compute that $\phi_{\bar C}(st) = st(P) - P = 
st(P) - t(P) + t(P) - P = \phi_{\bar C}(s) + \phi_{\bar C}(t)$.
\end{proof}

If $g \in N(S)$, then $g(\bar C) = \bar C'$ is isomorphic 
to $\bar C$ (via $g$) and thus $E(C')$ and $E(C)$ can be identified using $g$. 
We have the commuting diagram: 
$$
\begin{matrix}
I&\rightarrow&\bar C\\
\llap{$g$}\downarrow&&\downarrow\rlap{$g$}\\
I&\rightarrow&g(\bar C).
\end{matrix}
$$ 

Given such a $g$, let 
$\bar g \in N(S)/S = SL_2(F_3)$ be the image, which defines 
the action of $g$ on $E(I)$ and the action of $g$ on $S$ 
by conjugation. 
If $b: \bar C \to \P^2(\bar F)$ is an embedding, 
then $g$ defines an embedding 
$bg^{-1}: \bar C' = g(\bar C) \to \P^2(\bar F)$ and 
$g \circ \iota$ 
is the associated embedding $I \to \bar C'$. 

\begin{proposition} \label{prop:14} 
\begin{enumerate}[label=(\alph*)]
\item$\bar g \circ \psi = \psi \circ \bar g$. 

\item There is a choice of $\phi_{g(\bar C)}$ such that 
$\phi_{g(\bar C)} = \phi_C \circ \bar g$. 

\item Suppose $\bar C$ and $\bar C'$ are curves with 
isomorphic Jacobians which we identify via 
this isomorphsim. Let $\phi_{\bar C'}$ and 
$\phi_{\bar C}$ be choices of maps, as above, 
$S \to E(\bar C) = E(\bar C')$. 
If $\phi_{\bar C'} = s\phi_{\bar C}$ for 
$s$ in the automorphism group of $E(\bar C)$ then 
$\bar C' = \bar C$ as subvarieties of $\P^2$. 

\item $\phi_C: S \to E[3]$ preserves the pairing (where 
$E[3]$ has the Weil pairing). 
\end{enumerate}
\end{proposition}

\begin{proof}
 $(\psi \circ \bar g)(s) = \{(gsg^{-1}(P),P)\} = 
\{(gsg^{-1}(g(P)),g(P)\} = g(\{(s(P),P)\})$ which proves a). 
As for b), we use the embedding $g \circ \iota: I \to \bar C'$ 
to define $\phi_{\bar C'}$.
We have $\phi_{\bar C'}(s) = 
E(g \circ \iota)(\psi(s)) = 
E(g \circ \iota){(s(g^{-1}P), g^{-1}P)} = 
E(g)(s(g^{-1}P) - g^{-1}P) = 
gsg^{-1}P - P = \phi_C(gsg^{-1})$. 

If $E(\bar C)$ and $E(\bar C')$ are isomorphic 
then $\bar C \cong \bar C'$. Thus there is a 
$g \in N(S)$ such that $g(\bar C) = \bar C'$. 
Since $\bar g$ acts linearly on $S$ and $E(\bar C)$ 
we have $\phi_{\bar C} \circ \bar g = \bar g \circ \phi_{\bar C}$
Part c) follows because if 
$\bar g$ maps to $S^+/S$ then $g \in S^+$ which 
stabilizes $\bar C$. Part d) 
follows from the description of the Weil pairing in 
\cite{Si} p. 98.
\end{proof}

\section{General $F$} 

Now assume $F$ is a general field of characteristic not 2 or 3. Let $\bar F$ be the separable closure and $\bar G$ the Galois 
group of $\bar F/F$. it will be useful to consider 
fields $F \subset K \subset \bar F$ where $K/F$ 
is Galois with group $G$. If, say, $C$ is a structure 
defined over $F$ we let $C_K$ mean the extension of scalars. 
If $C \subset \SB(A)$  (and hence defined over $F$) then so is 
$E = E(C)$ and $I \subset \SB(A)$, but $E$ maybe defined 
over fields where $C$ is not. It is equally true that 
$I \subset \SB(A)$ may be defined over $F$ and $C_K  \supset I_K$ but $C$ not defined over $F$. 
Let  $C$ above have stabilizer $S^+$,  
and distinguished $S \subset S^+$. $G$ acts on 
$K \times_F \SB(A)$ and $K \times_F \SB(S^3(A)) = P^9_{K}$, as well as on $A_K^*$ and $A_K^*/K^*$. 
If $K = \bar F$ we can identify $A_K$ with 
$M_3(\bar F)$ and $A_K^*/K^*$ with $PGL_3(\bar F)$ 
and so $M_3(\bar F)$ $PGL_3(\bar F)$ has induced $\bar G$ actions 
but not the obvious ones, but rather actions 
where $M_3(\bar F)^{\bar G} = A$ and 
$(PGL_3(\bar F))^{\bar G} = A^*/F^*$. 
Similarly $\bar G$ acts on $\P^2(\bar F)$ 
with quotient $SB(A)$. If $\sigma \in \bar G$ 
and $g \in PGL_3(\bar F)$ and 
$P \in \P^2(\bar F)$ then $\sigma(g)(\sigma(P)) = 
\sigma(g(P))$. Recall the line bundle 
${\cal O}(1)_{\P^9}$ descended to a line bundle 
$L$ on $SB(A)$ and $\bar L$ on $\P^2(\bar F)$. 
Of course, the global sections of $\bar L$ 
have a $\bar G$ action whose invariants are the global sections of $L$. The curves we are studying 
are exactly the zeroes of global sections of 
$L$. 

If $C \subset SB(A)$, then 
$G$ preserves $I$ as a set. That is, $I$ is a zero dimensional 
reduced subvariety of $\SB(A)$. Moreover, $\bar G$ preserves 
$I$ as a set if and only if it preserves the associated 
$S \subset \PGL_3(\bar F)$ as a set (not 
elementwise). Since $\bar G$ preserves $C$, it 
preserves $S^+$, but may act nontrivially on 
$S^+/S$ because $\bar G$ may act nontrivially 
on the automorphism group of $E(C)$. Of course, 
if $j \not= 0, 1728$ then $S^+/S = C_2$ and $\bar G$ 
acts trivally on this.

We fix the embedding $I \subset \SB(A)$ and call 
it $a$. If $a$ is also the induced embedding 
$\bar I \subset \P^2(\bar F)$, and $\sigma \in \bar G$, 
then $\sigma \circ a = a \circ \sigma$. 
Since the set of lines in $\P^2$ are preserved 
by $G$, $G$ preserves colinearity and there is an induced map 
$G \to \Aff(I)$. 
If $S \subset PGL_3(\bar F)$ is preserved by $\bar G$, then 
$G$ also preserves the pairing $S \times S \to <\rho>$ where 
$\bar G$ might act nontrivially on $<\rho>$.

We want to use $\phi_C$ to understand how $G$ or $\bar G$ acts on the set of $C_K$. If $\bar G$ preserves 
$I$, then it permutes the set of $\bar C$'s containing $I$.  
Then $\bar G$ also acts on $E(I)$ and we have 
defined $\psi: S \to E(I)$ independent of any 
$\bar C$.  However, if $E(C)$ is defined 
over $F$, then $G$ acts on $E(C)$. We prove: 

\begin{proposition} \label{lem:15} Assume 
$I \subset SB(A)$ is such that $\bar I \subset 
\P^2(\bar F)$ is a set of inflection points for a 
smooth degree 3 curve $\bar C \subset \P^2(\bar F)$. 
\begin{enumerate}[label=(\alph*)]
\item $\psi \circ \sigma = \sigma \circ \psi$ 

\item If $C \subset SB(A)$ contains $I$ then 
$\sigma\phi_C\sigma^{-1} = \phi_C$ 

\item Suppose that for all 
$\sigma \in \bar G$, $\sigma\phi_C\sigma^{-1} = \phi_C \circ t$ for some $t \in S^+/S$ 
as set maps. Then $\bar C$ is $\bar G$ invariant and 
defines $C \subset SB(A)$. 
\end{enumerate}
\end{proposition}

\begin{proof}
$\psi \circ \sigma(s) = 
{(\sigma{s}\sigma^{-1}(\sigma(P), \sigma(P))} = 
\sigma({(s(P),P)}) = \sigma \circ \psi(s)$ which 
proves a). As for b), the assumptions imply that 
$\iota: I \to \bar C$ is preserved by $\bar G$ and 
b) follows from a). 

Turning to c), $\sigma\iota\sigma^{-1} = \iota'$ 
defines an embedding $I \to 
\bar C'$ where $\bar C' = \sigma(\bar C)$. 
If $b: \bar C \to \P^2(\bar F)$ is this embedding, 
then $a = b\iota = (\sigma{b}\sigma^{-1}) \circ 
(\sigma\iota\sigma^{-1})$ and $\sigma{b}\sigma^{-1}$ 
is an embedding for $C'$. 
Moreover, $\sigma\phi_C\sigma^{-1} = 
E(\sigma\iota\sigma^{-1}) \circ \psi$ and so 
$\sigma\phi_C\sigma^{-1} = t\phi_{C'}$ for some 
$t \in S^+/S$. Thus $C' = C$. That is, 
$\bar G$ preserves $\bar C$ as a set. 
If $q \in \bar L$ has zeroes the set $\bar C$, 
then $\bar G$ preserves $\bar Fq$ and hence, 
by Hilbert 90, there is a $q \in L$ 
whose zeroes $C \subset SB(A)$ define $\bar C$ over $F$.
\end{proof}

We can now state and prove our main theorem. 

\begin{theorem} \label{theom:16}
 Let $E$ be an elliptic curve over $F$. Let 
$A/F$ be a degree-three algebra. Let $K/F$ be the 
Galois extension with group $G$ obtained by adjoining all the 
three torsion points $E[3]$ to $F$. 
There is a $C \subset \SB(A)$ with $E(C) = E$ if and only if 
there is a subgroup $S \subset (A \otimes_F K)^*/K^*$ such that 
$S$ is preserved by $G$, $S$ is generated by $x,y$ with 
preimage $x',y'$ such that $x'y' = {\rho}y'x'$, and $S$ is 
isomorphic to $E(C)[3]$ as $\bar G$ modules with pairing.
\end{theorem} 

\begin{proof}
 We restrict ourselves to the $PGL_3(\bar F)$ 
orbit associated with $E$. 
If $C$ is as given, then $C$ defines $I \subset \SB(A)$ 
and so an associated $S \subset \PGL_3(\bar F)$. 
By Lemma~\ref{lem:15}, $\phi_C: S \to E[3]$ is a $G$ morphism. 
Note that, in particular, the elements of $S$ are fixed by the 
Galois group of $\bar F/K$ and so $S \subset (A \otimes_F K)^*/K^*$. 

Conversely, suppose $\phi: S \to E[3]$ is a $G$ isomorphism preserving the pairing. 
Let $S$ define $I$ and let $\bar C \supset I$. 
We would like to prove there 
is a $g \in N(S)$ such that $g(\bar C)$ 
comes from a curve in $SB(A)$ . 
Let $\bar g = \phi_C^{-1} \circ \phi \in SL_2(F_3)$ 
and lift $\bar g$ to $g \in N(S)$. Note that $\bar g$ has 
determinant one because $\phi$ preserves the pairing.  
Now $\phi = 
\phi_C \circ g = \phi_{g(C)}$ preserves the $\bar G$ 
action so 
we are done by Proposition~\ref{lem:15}.
\end{proof} 

Assume $I$ is defined over $F$. 
It is clear that the elements of $\bar G$ act 
as affine transformations of $I$ and so there is a 
homomorphism $\bar G \to \Aff(I)$. 
Let us consider the possible images of $\bar G$ 
in $\Aff(I)$. Set $G_0$ to be the image of $\bar G$ 
in $GL_2(F_3)$. If $I \subset C$ and $C$ is defined over 
$F$ then $G_0$ is the Galois group of $K/F$ 
where $K$ is the field defined by adjoining all the 
three torsion points of $E(C)$ and so if $\bar H$ 
is the kernel of $\bar G \to G_0$, then $\bar H$ 
is the Galois group of $\bar F/K$.   
 
If the primitive third root $\rho$ is in $F$, 
then $G_0 \subset SL_2(F_3)$. Otherwise, the 
image of $G_0 \to GL_3(F_3) \to F_3^*$, 
the second map being the determinant, is 
the Galois group of $F(\rho)/F$. 
Other than this, we cannot restrict $G_0$ in 
any way. 

Of course $\bar H$ maps to $S \cong E(I)$ with image we call 
$H_0$. If $\bar J$ is the kernel of $\bar H \to H_0$, 
then $\bar F^{\bar J}$ is $K(I)$, obtained by adjoining the inflection points themselves to $K$. 
Note that since $S$ acts by translation, adjoining one point 
of $I$ over $K$ is equivalent to adjoining them all. 

Suppose $I$ corresponds to $S \subset A_K^*/K^*$. Let $x,y \in S$ 
be a basis and lift to $\tilde x, \tilde y$. Then $I$ consists 
of $L_x \cap L_y$ where $L_x \subset \bar F^3$ is a two dimensional 
space preserved by $\tilde x$, and similarly for $L_y$. 
Summing over the $L_y$, we have that all $L_x$ (and $L_y$) 
are defined over $K(I)$. That is $\tilde x$ and $\tilde y$ 
have degree three equations which split. We will say, in this case, 
that $x$ (and $y$) splits. 

We can say this another way. Taking cubes defines an injective  
morphism $S/K^* \to K^*/(K^*)^3$. If $T \subset K^*/(K^*)^3$ 
is the image, we can form the field $K(T^{1/3})$ of all 
cube roots and let $T^{1/3} \subset K(T^{1/3})$ be the cube roots 
of all elements in $T$. By Kummer theory the Galois group 
of $K(T^{1/3})/K$ is $\Hom(T^{1/3}/K^*,\mu_3)$. The cube 
and cuberoot maps compose to form a well defined homomorphism 
$S \to T^{1/3}/K^*$. Using the pairing on $S$ and duality 
we get a homomorphism $\Hom(T^{1/3}/K^*,\mu_3) \to S$ 
and the composition $\bar H \to \Hom(T^{1/3}/K^*,\mu_3) \to S$ 
can be easily seen to be the map above.  

\begin{lemma} 
$K(I) = K(T^{1/3})$. The Galois group of 
$K(I)/K$ has degree 1,3, or 9 and if it is not 9 
then $A_K \cong M_3(K)$. If $K(I)/K$ has degree 3 
then no point of $I$ is rational over $K$ none the less. 
\end{lemma} 

\begin{proof}
Over $K(T^{1/3})$ both $x$ and $y$ split and 
so $I$ is defined over this field. Conversely, $K(T^{1/3})$ 
is clearly the smallest field where $x$ splits for 
any $x \in S$. The Galois group of $K(I)/K$ is a 
subgroup of $S$. If $K(T^{1/3})/K$ has degree 1 or 3, 
then for some nontrivial $\tilde z = \tilde x^a\tilde y^b$ 
we have $\tilde z^3$ is a cube which forces the splitting 
of $A_K$. 
\end{proof} 

Let $G_1$ be the image of $\bar G$ in $\Aff(I)$, 
so we have an exact sequence $1 \to H_0 \to G_1 \to G_0 \to 1$ 
which is induced by the split exact sequence $1 \to E(I) \to 
\Aff(I) \to GL_2(F_3) \to 1$. In fact the first sequenece also 
splits which is almost but not quite immediate. If $H_0$ is trivial 
there is nothing to prove, and if $H_0 \cong E(I)$ the splitting 
is immediate. If $H_0$ has order three then this extension defines 
an $\alpha \in H^2(G_0,H_0)$ and it also is the case 
that $\alpha = 0$ as follows. If nontrivial, $C_3$ is the 
3 - Sylow subgroup of $G_0$ and it suffices to observe 
that $\alpha$ restricts to 0 in $H^2(C_3,H_0)$. This follows 
because one can check that $\Aff(I)$ has no elements of order 9. 

We further investigate the possibilities. In general we know that 
$G_0$ can be any subgroup of $GL_2(F_3)$, and if $H_0 = 1$ 
or $H_0 \cong E(I)$ 
there is nothing more to say. 
When $H_0$ has order 3 we can restrict $G_0$ further. 
Then $\tilde S$ is generated by $\tilde s$ 
and $\tilde y$ where $\tilde y$ is split and $\tilde x$ 
is not. Clearly $A_K$ is split. 
$G_0$ acts on $H_0$ via conjugation and clearly 
$G_0$ preserves the cyclic subgroup generated by $y$. 
If $\rho \in F$, 
this implies that $G_0 = {1}$, $G_0$ is a subgroup of the cyclic group $C_6$. If $\rho \notin F$, 
then $G_0$ is a subgroup of $G_0 = S_3 \oplus C_2$ 
containing the central subgroup $C_2$.

We want to use Theorem \ref{theom:16} to describe the 
$C$ that appear for a $E(C)$ given by that theorem. 
Recall that any such $C$ is a principal homogeneous space 
over $E(C)$, and all such spaces are classified 
by the Galois cohomology group $H^1(\bar G,E(\bar C))$. 
Note that this is the correct cohomology group because 
$E(\bar C)$ is the group of automorphisms of $\bar C$ 
as a principal homogeneous space over $E(\bar C)$. 

More precisely, let $\bar C$ be such a space 
and let $P$ be a point on $\bar C$. Define $e: \bar G \to 
E(\bar C)$ by setting $e(\sigma) = \sigma(P) - P \in E(\bar C)$. 
It is easy to check that $e$ is a one cocycle. 
Note that if we pick another point $Q$ on $\bar C$ 
and use that to define $e': \bar G \to E(\bar C)$, 
then $e(\sigma) - e'(\sigma) = \sigma(\alpha) - \alpha$ 
for $\alpha = P - Q$ the element of $E(\bar C)[3]$. 
Thus $\bar C$ defines a class in 
$\gamma' \in H^1(\bar G, E(\bar C))$. 

In our situation it makes sense to define the cycle 
using an inflection point $P \in I \subset \bar C$. 
Furthermore let $\phi: S \to E(\bar C)[3]$ be the 
Galois preserving isomorphism. Note that $\phi$ is a 
$\phi_C$ from the previous section, so there is an 
embedding $I \subset \bar C$ such that $\phi(s) = 
s(P) - P$ for any choice of $P$. It is clear that 
any $\sigma \in \bar G$ defines an affine map on $I$, 
so we have a homomorphism $\Phi: \bar G \to \Aff(I)$. 

The composition $\bar G \to \Aff(I) \to GL_2(F_3)$ 
defines the action of $\bar G$ on $E(I)$ and hence 
$E(\bar C)[3]$ and hence on $S$. Let 
$\gamma \in H^1(\Aff(I),E(I))$ be the canonical class 
defined in Proposition~\ref{gamma} that It is clear from the 
above description of $\gamma$

\begin{theorem} $\Phi^*(\gamma) \in H^1(\bar G, E(\bar C)[3]))$ 
maps to $\gamma' \in H^1(\bar G,E(\bar C))$ and $\gamma'$ defines 
$\bar C$ as a principal homogeneous space over $E(C)$. 
\end{theorem}

The connection between $\bar C$ and the algebra $A$ is 
stated in \cite{O2} and briefly recalled next. 
Since $S \subset PGL_3(\bar F)$ we have a composition 
$\eta: H^1(\Aff(I),E(I)) \to H^1(\bar G,E(\bar C)[3]) \to H^1(\bar G,PGL_3(\bar F)) \to 
H^2(\bar G,\bar F^*)$ where, we recall, the last map 
is not a homomorphism because $H^1(\bar G,PGL_3(\bar F))$ 
is not a group. 

\begin{proposition}
$\eta(\gamma) \in H^2(\bar G,F^*)$ defines the Brauer class 
of $A/F$. Thus if $C \subset \SB(A)$ then $C$ is defined 
by the image of a $\gamma' \in H^1(\bar G,E(\bar C)[3])$ 
which maps to the Brauer class of $A/F$. 
\end{proposition} 

\begin{proof} We have a given action of $\bar G$ 
on $M_3(\bar F)$ such that the invariant ring is $A$ 
and such that the quotient of the action on $\P^2(\bar F)$ 
is $\SB(A)$. 
If $C \subset \SB(A)$ and $P \in I \subset \bar C$, 
then $P$ defines the cocycle $e(\sigma): \bar G \to E(\bar C)[3])$. 
Then $e(\sigma)^{-1}\sigma$ defines a new action on $\P^2(\bar F)$ 
which fixes $P$ and hence has quotient variety $\P^2(F)$, 
so new invariant ring $M_3(F)$.
\end{proof}

\section{Examples}

Now we start giving examples, including examples where a $E(C)$ does not appear. We start with the easiest corollary of Theorem 
\ref{theom:16}. Let $E$ be an elliptic curve defined 
of $F$ and $K \supset F$ be the field gotten by adjoining 
all the three torsion points of $E$. Assume $K = F$ or 
$K = F(\rho)$ where $\rho$ is a third root of one. 
The next result is proven in \cite{O2} and \cite{O3} 
for the $K = F$ 
case and in the preprint \cite{AA} for the $K = F(\rho)$ case, 
by different means. 

\begin{corollary} \label{cor:18}
 If $E/F$ 
is an elliptic curve where all three torsion 
points are defined over $F(\rho)$, and $A/F$ 
is a degree-three central simple algebra, 
then there is an $C \subset \SB(A)$ such that 
$E(C) = E$. 
\end{corollary}

\begin{proof} When $K = F$ this is immediate because 
 $A/F$ is cyclic. When $K = F(\rho)$ we note the following. 
First, if $\sigma$ generates the Galois group of $K/F$ 
then $E$ has a three torsion point $P$ such that $\sigma(P) = P$ 
and a three torsion point $Q$ such that $\sigma(Q) = -Q$. 
If $L \supset K$ is such that $L/F$ is dihedral, then 
$L = L' \otimes_F K$ where $L' = F(a^{1/3})$ and $\sigma$ fixes 
$a^{1/3}$. If $M'/F$ is cyclic of degree 3 and $M = M' \otimes K$ 
then $M = K(b^{1/3})$ where $\sigma(b^{1/3}) \in b^{-1/3}F^*$. 
Again, $E$ is the Jacobian of some $C \subset SB(A)$ 
by \ref{theom:16} and the fact that $A/F$ is cyclic.
\end{proof}

Let $K$ still be the extension field 
of $F$ obtained by adjoining 
all the three torsion points of $E(C)$.  
To begin with, assume $K/F$ is cyclic of degree three. If $C$ exists, 
there is an $S \subset (A \otimes_F K)^*/K^*$ generated by 
images of $x',y' \in A \otimes_F K$ such that $x'y' = {\rho}y'x'$ 
and $S$ is preserved by $G = <\sigma>$. It follows that we can 
choose $x',y'$ such that $\sigma(y') \in y'K^*$ and 
$\sigma(x') \in x'y'K^*$. Suppose $\sigma(y') = y'z$. Since $\sigma$ 
has order three on $A \otimes_F K$, we have $z\sigma(z)\sigma^2(z) = 1$ 
and so $z = \sigma(u)/u$ for $u \in K^*$. Replacing $y'$ by $y'u^{-1}$ 
we may assume $\sigma(y') = y'$ or $y' \in A$. Now write 
$\sigma(x') = x'y'w$ for $w \in K^*$. Again using that $\sigma$ has order 
three, we have $x' = \sigma^3(x') = \sigma^2(x'y'w) = 
\sigma(x'y'wy'\sigma(w)) = x'y'wy'\sigma(w)y'\sigma^2(w)$ 
or $1 = y'^3 w\sigma(w)\sigma^2(w)$. That is, 
$A$ has the form $(a,b)$ where $(a,K/F)$ is trivial. Said another 
way, $A$ is split by an $F(a^{1/3})$ such that $(a,K/F)$ is split. 

The above condition is actually necessary and sufficient. 

\begin{proposition} \label{prop:17}
 Let $A/F$, $E$ and $K/F$ be as above. 
Then there is a $C \subset \SB(A)$ with Jacobian $E$ if and only 
if $A$ is split by $F(a^{1/3})$ where $(a,K/F)$ is trivial.
\end{proposition} 

\begin{proof}
 We saw one direction above. Suppose $A$ is split by $F(a^{1/3})$ 
and $a = N_{K/F}(w)$. Then $N_{K/F}(a/w^3) = a^3/N_{K/F}(w)^3 = 1$. 
Thus there is a $b' \in K$ with $\sigma(b')/b' = a/w^3$. 
Let $B'/K$ be the degree-three cyclic algebra generated by $x'$, $y'$ 
where $x'^3 = a$, $y'^3 = b'$, and $x'y' = {\rho}y'x'$. 
The action of $\sigma$ on $K$ extends to $B'$ by setting 
$\sigma(x') = x'$ and $\sigma(y') = y'x'w^{-1}$. Then as above 
$\sigma^3(y') = \sigma^2(y'x'w^{-1}) = \sigma(y'x'w^{-1}x'\sigma(w^{-1})) = 
y'x'w^{-1}x'\sigma(w^{-1})x'\sigma^2(w^{-1}) = y'aN_{K/F}(w^{-1}) = y'.$ 
It follows that this extension of $\sigma$ has order 3 on $B'$ and thus 
the invariant algebra $B$ is central simple of degree three over 
$F$. Since $B$ contains $x'$ it is split by $F(a^{1/3})$. 
Consider $B^{\circ} \otimes A$ which is also split by $F(a^{1/3})$ 
and hence is represented by an algebra $A' = (a,c)$ over $F$. 
Then $A \otimes_F K$ is similar to 
$D = B' \otimes _K (a,c) = (a,b')_K \otimes (a,c)_K \sim 
(a,b'c)_K$. More precisely, let $e \in F(a^{1/3}) \otimes_F F(a^{1/3})$ 
be the separating idempotent. Then $e$ can be viewed as an element 
of $D$, $A_K = eDe$, and $e$ is $\sigma$ fixed. 
Let $x'', y'' \in (a,c)$ be the elements with $x''^3 = a$, $y''^3 = c$, 
and $x''y'' = {\rho}y''x''$. Set $\alpha = e(x' \otimes 1)e$, 
$\beta = e(y' \otimes y'')e$, both elements in $eDe = A_K$ (but remember 
$e$ is the identity there). Clearly $\alpha^3 = a$ because $e$ commutes 
with $x'$. For the same reason, $\alpha\beta = \rho\beta\alpha$. 
Also, $\sigma(\beta) = e(\sigma(y') \otimes \sigma(y''))e = 
e(y'x'w^{-1} \otimes y'')e = (e(y' \otimes y'')(x'w^{-1} \otimes 1))e = 
e(y' \otimes y'')ee(x'w^{-1} \otimes 1)e = \beta{\alpha}w^{-1}$. 
If $S$ is the image of $\alpha, \beta$ in 
$(A \otimes_F K)^*/K^*$ we are done.
\end{proof}

As a route to a counter example (that is a case where $E(C)$ 
does not appear), let $F$ 
be a field complete with respect to a discrete 
valuation ring with prime $\pi$. 
Suppose $K/F$ ramifies at a prime $\pi$ 
and $A/F$ is a division algebra unramified at 
$\pi$. 
Then any norm of $K/F$ has the form $u^3\pi^r$ and if $r$ is prime 
to $3$, $F(\pi^{r/3})$ cannot be a subfield of the unramified $A$. Thus if $K/F$, and $E$ are as in Proposition~\ref{prop:17}, $E$ cannot appear as an $E(C)$ for $C \subset 
\SB(A)$. More generally, if $F$ is a discrete 
valued field 
with prime $\pi$ and the above hold over the 
completion $F_{\pi}$, then again $E$ does not 
appear. For a specific example, according to Bruce Jordan, let $E$ be the elliptic curve $y^2 + y = x^3 - x^2 - 10x - 20$. This is $X(11)$, If $K/\Q$ is obtained by adjoining the three torsion points, then 
$K/\Q$ ramifies over $\Q$ only at 3, and 11 and with ramification 
index 3 or 6. Thus there is an $F \supset \Q$ 
such that $K/F$ is cyclic of degree three and ramifies 
at a prime $p$ extending 11. Note that $F$ 
must contain $\rho$. Furthermore over 
$K_p/F_p$ is ramified and obtained by adjoining 
the three torsion points of $E$ to $F_p$. 
Now the $p$-adic fields have no 
unramified division algebras but if we adjoin 
an indeterminate we can create one. The degree-three 
algera $A = (K_p(t)/F_p(t),t)$ over $F_p(t)$ is unramified and unsplit 
at $p$. It follows that $E$ does not appear as 
$E(C)$ for a $C \subset \SB(A)$.  

One might hope that when $K/F$ is more general, 
there might be a result similar to Proposition 17, 
that is a criterion with a cohomological flavor. The next 
result makes clear there can be an obstruction 
of a more arithmetic nature, involving the 
non-appearance of the dihedral group of order 6 
as a Galois group. 

\begin{corollary} \label{cor:19}
Suppose $E/F$ is an 
elliptic curve and $K/F$ is the field obtained 
by adjoining all the three torsion points of 
$E$. Assume $K/F$ is cyclic of order 2 
and $F$ contains $\rho$. 
Let $D/F$ be a division algebra of degree three. 
Then there is a $C \subset \SB(D)$ with 
$E(C) = E$ if and only if there is a field 
$L/K$ splitting $D$ such that $L/F$ is Galois 
with group $S_3$, the dihedral group of order 
6. 
\end{corollary}

\begin{proof}
 Let $\sigma$ generate the Galois group of 
$K/F$. Assume $S \subset (D \otimes K)^*/K^*$ exists as 
in Theorem~\ref{theom:16}. Then $\sigma$ acts as $-1$ 
on $S$. If $L/K$ is any of the cyclic field 
extensions coming from $S$, then $L/F$ 
is dihedral Galois. 

Conversely, suppose such an $L/K$ 
exists. Then $D \otimes_F K = (a,b)_K$ 
where $\sigma(a) = a^{-1}z^3$ for $z \in K^*$. 
Since $D$ is fixed by $\sigma$, 
$\sigma(b) = b^{-1}N(u)$ where $u \in L^*$ 
and $N: L^* \to K^*$ is the norm map.  
There is a surjection 
$K^*/(K^*)^3 \to K^*/N(L^*)$ which is morphism 
of modules over the group ring $F_3[<\sigma>]$. 
Let $\bar b \in K^*/N(L^*)$ be the image of 
$b$. Since this group ring is semisimple, 
there is a preimage $\tilde b \in K^*/(K^*)^2$ 
such that $\tilde b$ maps to $\bar b$ and 
$\sigma(\tilde b) = (\tilde b)^{-1}$. 
That is, there is a $b' \in K^*$ such 
that $b'$ maps to $\tilde b$ which maps 
to $\bar b$. That is, $b' = bN(u')$ for some 
$u' \in L^*$ and $\sigma(b') = b'^{-1}w^3$ 
for some $w \in K^*$. We can write 
$D \otimes_F K = (a,b')_K$ and so 
$D \otimes_F K$ contains $x,y$ with 
$x^3 = a$, $y^3 = b$, $\sigma(x) = x^{-1}z$, 
$\sigma(y) = y^{-1}w$ and $xy = \rho{yx}$. 
If $S \subset (D \otimes_F K)^*/K^*$ is generated 
by the image of $x$ and $y$, then 
$S$ is as needed.
\end{proof}

The above result can be used in two ways to get interesting examples. 
Let $K$ be a finite extension of some $\Q_p$ 
and $E$ an elliptic curve over $K$ with the following properties. 
First, suppose $p$ is prime to three and $K$ contains a primitive third 
root of one. In addition, assume all the three torsion points 
of $E$ are rational over $K$. It follows that there is no $L/K$ an extension of fields such that the Galois group is $S_3$. Let $K'/K$ 
be a quadratic extension of fields and $E'/K$ the corresponding 
quadratic twist. Then $\Gal(K'/K)$ acts on the three torsion 
points as $-1$. Let $D/K$ be either of the degree-three division 
algebras. $D \otimes_K K'$ cannot have the required $S$ because of the following. Suppose $S$ existed, 
and let $L/K'$ be a cyclic extension defined by 
a rank one subgroup of $S$. Then $L/K$ would 
have Galois group $S_3$, a contradiction. 
Thus $E'$ does not appear as an $E(C)$ for 
$C \subset \SB(D)$ but $E$ does. 

The above is an example where there are two elliptic curves with the same $j$ invariant 
where only one of the them appears as an 
$E(C)$ for $C \subset \SB(D)$. For a different 
kind of example,
assume $F$ is a number field containing 
$\rho$ and $E/F$ 
is an elliptic curve such that the field, $K$, 
formed by adjoining all the three torsion 
is such that $K/F$ is cyclic Galois 
of degree four with $<\sigma>$ as Galois group. 
Note that this property is also true for all 
the quadratic twists of $E$ over $F$. 
By density, there are infinitely 
many primes $p$ of $F$ such that $K_p = 
K \otimes_F F_p$ is a field where $F_p$ 
is the completion. For all but finitely 
many of these primes $p$, for all finite 
$K' \supset F_p$ there is no Galois extension of fields $L/K'$ with Galois group $S_3$. 
Assume $D \otimes_F F_p$ is a division algebra 
for one of these primes $p$ and $K_p'$ 
is such that $F_p \subset K_p' \subset K_p$ 
and $K_p/K_p'$ has Galois group $<\sigma^2>$. 
If $S \subset D^*/F^*$ existed as in Theorem 16, 
the same would be true for 
$(D \otimes_F K_p)^*/K_p^*$. Since $\sigma^2$ 
acts on $S$ as $-1$, this is a contradiction. 
Hence for such $D$ and $E$, there are no 
$C \subset \SB(D)$ such that $E(C)$ has the 
same $j$ invariant as $E$.

\end{document}